\documentclass[12pt]{article}
\usepackage[english]{babel}
\usepackage[utf8]{inputenc}
\usepackage[T2A]{fontenc}
\usepackage{graphicx}
\usepackage{hyperref}
\usepackage{amsmath, amstext, amssymb, amsthm, amsfonts, dsfont, color}
\usepackage{multicol}
\usepackage{array}

\newcommand{\eps}{\varepsilon}
\newcommand{\Fi}{\phi}

\newcommand{\bbR}{\mathbb{R}}
\newcommand{\bbZ}{\mathbb{Z}}

\newcommand{\bbN}{\mathbb{N}}

\newcommand{\Pb}{\mathbb{P}}
\newcommand{\om}{\omega}

\newcommand{\Homeo}{\mathrm{Homeo}}
\newcommand{\supp}{\mathop{\mathrm{supp}}}
\newtheorem{Th}{Theorem}
\newtheorem{Lem}{Lemma}
\newtheorem{pro}{Proposition}

\theoremstyle{definition}
\newtheorem{definition}{Definition}
\newtheorem{notation}{Notation}
\newtheorem{rem}{Remark}

\author{A. Gordenko
\thanks{The author's work was partially supported by ANR Gromeov (ANR-19-CE40-0007).} 
\thanks{Univ Rennes, CNRS, IRMAR - UMR 6625, F-35000 Rennes, France.}
}
\title{Random dynamical systems on a real line}

\begin{document}

\maketitle

{\abstract{We study random dynamical systems on the real line, considering each dynamical system together with the one generated by the inverse maps. We show that there is a duality between forward and inverse behaviour for such systems, splitting them into four classes (in terms of both dynamical and stationary measure aspects). This is analogous to the results already known for the smooth dynamics on $[0,1]$, established in terms of the Lyapunov exponents at the endpoints; however, our arguments are purely topological, and thus our result is applicable to the general case of homeomorphisms of the real line.}}

\section{Introduction.}

This paper is devoted to the study of random dynamic systems (RDS) on the real line. That is, we are given a finite number of homeomorphisms $f_1, \dots, f_k \in Homeo^+(\bbR)$ together with the probabilities $p_1, \dots, p_n$ of their application. On each step we apply one of these maps, chosen independently in accordance to these probabilities; the reader will find precise details in Section \ref{def} below.

This work was motivated by the paper of Deroin et al.~\cite{DKNP}, where the authors have considered the case of {\bf symmetric} dynamics, that is, applying any map with the same probability as its inverse. They have shown that in the symmetric case, except for some degenerate situations, there is no \emph{probability} stationary measure (we recall the definition below), though there is an infinite Radon one. At the same time, the symmetric dynamics is always recurrent: there exists a compact interval such that a random orbit, starting from any point, almost surely visits this interval infinitely many times. However, the symmetry in~\cite{DKNP} was used in an essential way and it is interesting to study all the possible types of behavior when this assumption is omitted.

Note, that a change of coordinates transforms $\bbR$ into the interval $(0,1)$. 
The dynamics on the interval and on the real line was studied by many authors, including Guivarc'h, Le Page~\cite{GP}, Deroin, Navas, Parwani~\cite{DKNP}, Ghaeraei, Homburg~\cite{Homb}, Brofferio, Buraczewski, Czudek, Czernous, Damek, Szarek, Zdunik,~\cite{BBD,Sar,CzS,SZ}, Alsed\`a, Misiurewich~\cite{AM}, Kan~\cite{Kan}, Bonifant, Milnor~\cite{BM}, Ilyashenko, Kleptsyn, Saltykov~\cite{IKS}, and many others.

In many of these works, their authors have studied RDS on $(0,1)$ under under additional smoothness (and minimality) assumptions: e.g. in \cite{Homb, IKS, Kan} it is assumed that dynamics is smooth everywhere, in \cite{CS,SZ}~--- at the endpoints. This smoothness assumption has allowed to invoke the technique of the Lyapunov exponent to describe the behaviour at the endpoints.

Namely, it is quite natural to expect --- and the authors, mentionned above, have shown it -- that positive random Lyapunov exponents at the endpoints imply the ``random repulsion'' and thus a probability stationary measure, supported inside the interval. On the other hand, negative Lyapunov exponents imply that the trajectories almost surely tend to endpoints. Finally, \emph{zero} Lyapunov exponents are somewhat close the the positive ones: a random orbit almost surely leaves the neighborhood 
such an endpoint, but the expectation of time to do so is infinite.

The first two types of behaviour are dual to each other in the following sense. Let us denote by $\mu$ the discrete probability measure on $\Homeo_+(\bbR)$ defining the dynamics (that is $\mu(\{f_i\}) = p_i$), and by $\hat \mu$ its image when all the maps are replaced by their inverses, so $\hat \mu(f) = \mu(f^{-1})$. We call the former \emph{forward dynamics}, the latter \emph{inverse}.
If for the forward dynamics the Lyapunov exponents are positive, then for the inverse one they are negative. Also, the inverse dynamics to the one with zero Lyapunov exponent also has zero Lyapunov exponent at that endpoint.  This allows to describe possible behaviours for the forward and backward dynamics, grouping these in quite a few classes.

\subsection{Main results}

In this paper, we show that that such conclusions (and a duality between forward and backward dynamics) can be established with no smoothness assumptions at all, by direct application of purely \emph{topological} methods. In the first result, Theorem~\ref{th:class} below, we show that for a random dynamical system on $\bbR$ the behaviours of forward and inverse dynamics fall into one of four ``dual'' classes.

\begin{Th}\label{th:class}
Assume that RDS on $\bbR$, defined by a finitely supported measure~$\mu$ on $\Homeo_+(\bbR)$ is such that 
\[
\forall x \in \bbR \quad  \exists f, g\in \supp \mu : \quad g(x) < x < f(x).
\]

Then, possibly upon interchanging $\mu$ and $\hat \mu$ and (or) reversing the orientation by a space symmetry $x\mapsto (-x)$, the action falls in exactly one of the following classes:
\begin{enumerate}
\item\label{i:1} in forward dynamics all the points almost surely tend to  $+\infty$, in inverse dynamics all the points almost surely tend to  $-\infty$;
\item\label{i:2} in forward dynamics all the points almost surely tend to  $+\infty$, the inverse dynamics is recurrent (all the points almost surely return to some compact infinitely many times);
\item\label{i:3} both forward and inverse dynamics are recurrent;
\item\label{i:4}  in forward dynamics all the points tend with positive probability to each of $+\infty$ or $-\infty$, the inverse dynamics is recurrent.

\end{enumerate}
\end{Th}  
Actually, the ``finitely supported'' assumption can be weakened to the ``compact displacements'' one (see Definition \ref{displace}). Moreover, part of the conclusions survive if we drop it completely (see Theorem~\ref{th:fi}). However, the dynamics in the infinitely supported case can behave much nastier. 
Namely, in Section~\ref{monster} we construct a monster, illustrating non-recurring dynamics, that does not tend individually neither to $+\infty$, nor to $-\infty$.

Our second result is devoted to the description of (Radon) stationary measures in the recurrent parts of these cases. The existence part essentially follows the construction in~\cite{DKNP}, however the interesting part is that these measures might be finite, infinite or semi-infinite~--- as well as their relation to the dynamics. Also, note that under an addition assumption of~\emph{proximality} of the action (that is, an arbitrary large interval can be contracted inside a given one), a recent result of Brofferio, Buraczewski and Szarek~\cite[Theorem~1.1]{Sar} implies that the Radon stationary measure is unique.

\begin{Th}\label{th:stat}
Let $\mu$ be a finitely supported probability measure on $\Homeo_+(\bbR)$, satisfying the assumptions of~Theorem~\ref{th:class}. Depending on into which of the four classes, described in Theorem \ref{th:class} does it fall, we have one of the following corresponding conclusions:
\begin{enumerate}
\item Both forward and backward dynamics are non-recurrent.
\item The forward dynamics is non-recurrent. The backward dynamics is recurrent and admits a \emph{semi-infinite} Radon stationary measure: the measure of half-rays to $+\infty$ is finite measure, of half-rays to $-\infty$ is not. This measure can be constructed using hitting probability for the forward dynamics. 
\item Both forward and backward dynamics are recurrent and admit an infinite Radon stationary measure and do not admit neither a probability, nor a semi-infinite one (the same conclusion as in the symmetric case);
\item The forward dynamics is non-recurrent; there is a probability stationary measure for the backward dynamics, its distribution function is the probability for a point to tend to~$+\infty$.
\end{enumerate}
\end{Th}

Grouping the conclusions of Theorems~\ref{th:class} and~\ref{th:stat}, we get Table~\ref{t:properties}.
\begin{table}[!h]
\begin{tabular}{|c|c|c|}
\hline
no. & Forward dynamics & Backward dynamics \\
\hline
\ref{i:1} & Everything tends to $+\infty$  & Everything tends to $-\infty$ \\
\hline
\ref{i:2} & Everything tends to $+\infty$  & The dynamics is recurrent and admits \\
& &  a semi-infinite stationary measure \\
\hline
\ref{i:3} & The dynamics is recurrent and admits & 
The dynamics is recurrent and admits \\
&  an infinite stationary measure & an infinite stationary measure\\
\hline
\ref{i:4} & Every point tends to $+\infty$ or to $-\infty$,   & The dynamics is recurrent and admits  \\
& to both with positive probabilities & a probability stationary measure \\
\hline
\end{tabular}
\caption{Possible cases for the dynamics}\label{t:properties}
\end{table}

\subsection{Plan of the paper}

We introduce the notations and recall the definitions in Section~\ref{def}. Then, in Section~\ref{fis}, we study the property of the functions $\phi_{+}$ and $\phi_-$, giving the probability for the images of the initial point~$x$ to tend to $+ \infty$ and $-\infty$ respectively. We then apply it in Section~\ref{2dyn} to study the possible behaviours for forward and backward dynamics simultaneously. Section~\ref{monster} is devoted to the construction of the monster example with points evading to infinity while oscillating between plus and minus infinities. Finally, Section~\ref{measures} is devoted to the constructions and study of stationary measures.

\section{Definitions and notation.}\label{def}

Let $\mathcal{F} = {f_1, f_2, \dots}$ be a finite (or infinite) set of <<sample>> elements of $Homeo^+(\bbR)$ with a probability measure $\mu$ on it. Let $\{g_n\}_{n=1}^{\infty}$ be a sequence of i.i.d. random variables, taking values in $\mathcal{F}$ and distributed in accordance with measure $\mu$. In finite (and countable) case it is convenient to have special notations for elementary probabilities; we'll denote these 
\[
\Pb(g_n = f_k) = \mu(f_k) := p_k.
\]

Consider the probability space $\Omega := (F^{\bbN}, \mu^{\otimes \bbN})$; in these terms, $g_n$ is a $n$th coordinate of $\omega \in \Omega$. Set
$$
F_n = F_{n,\om} = g_n \circ \dots \circ g_1,
$$
the left random walk on group $G = \langle \mathcal{F} \rangle$. Finally, 
$$
X_n(x) = F_n(x)
$$
is the Markov chain, defined for any $x \in \bbR$.

All above defines the RDS, to which we would refer as \emph{forward} from now~on. 

The \emph{inverse} dynamics is defined in the same way for $\hat {\mathcal{F} }= \{f_1^{-1}, f_2^{-1}, \dots\}$, with the corresponding measure $\hat \mu$ defined by
\begin{equation} \label{meas-trans}
\hat \mu (f_k^{-1}) = \mu(f_k) = p_k.
\end{equation}
It is convenient to add to the considered $\mathcal{F}$ the set of all $f_k^{-1}$ (with $\mu(f_k) = 0$ for any $f_k$ that wasn't there originally) for it to become more symmetric. This allows us to rewrite \eqref{meas-trans} as
\begin{equation} \label{meas-trans1}
\hat \mu (f_k) = \mu(f_k^{-1}).
\end{equation}

Intuitively, we can think of the inverse dynamics in three different ways. First is quite direct: instead of each $f_k$ we've taken its inverse $f_k^{-1}$, thus it is indeed {\it inverse} dynamics. Second, assuming that the set~$\mathcal{F}$ already contains each map together with its inverse, it is not changed by this <<inversification>>, but the probabilities are swapped between each $f_i$ and $f_i^{-1}$ (thus making it the dynamics with the same generating set, but with another, <<inverted>>, measure~$\mu$). Thirdly, we may think of the inverse dynamics as of the forward one with inverted time, thus making it a very natural object to investigate. However, one must note that if for a fixed time $n$ the law of $F_n$ for the inverse dynamics coincides with the law of inverse maps of $F_n$ for the forward one, their \emph{evolution} does not (as the order is composition is also inverted by the passing to the inverse).

As noted previously, we do not ask much of any $f_k$. However, we expect the whole RDS to hold the following property.

\begin{definition}\label{unbound}
We call the point $x \in \bbR$ \emph{shiftable}, if for any $a\in \bbR$ there're exist $k \in \bbN$ such that probabilities $\Pb(F_k(x) < x - a)$ and $\Pb(F_k(x)> x + a)$ are non-zero. Commonly speaking, it means that we can move $x$ arbitrarily far to the left and to the right with non-zero probability in finite amount of time. We say that RDS has the \emph{shiftability} property if any point $x\in \bbR$ is shiftable. 

It is equivalent to the existence of $f_{i_1}$ and $f_{i_2}$ in $\mathcal{F}$  for any fixed $x$ such that $f_{i_1}(x) < x < f_{i_2}(x)$. In work \cite{Sar} this is called  \emph{unboundedness}.
\end{definition}

Now we prove the following auxiliary result:
\begin{Lem}\label{limits}
Let $(\mathcal{F}, \mu)$ be RDS with shiftability property. Then for any $x\in \bbR$ with probability~1 the limits $\limsup_{n\to\infty}{F_n(x)}$ and $\liminf_{n\to\infty}{F_n(x)}$ are infinite.
\end{Lem}
\begin{proof}
Let us show by contradiction that for every finite interval $I\subset \bbR$ the probability that the upper limit $\limsup_{n\to\infty} F_{n}(x)$ takes value in $I$ is equal to~$0$. Indeed, assume the contrary, that for some $x\in \bbR$ and $I=(a,b)$ the probability of the event 
\[
A=\{\omega\in\Omega \mid \limsup{F_{n,\omega}(x)} \in I\}
\] 
is strictly positive. Note that $A\subset \bigcup_k A_k$, where 
\[
A_k = \{ \omega \mid \forall n\ge k \quad F_n(x) < b , \quad \text{and } \, F_n(x)> a \, \text{ infinitely often} \}
\].
Hence, for some $k$ the event $A_k$ has a positive probability. We fix such~$k$.

Now, the shiftability property implies that there exists a composition $G = f_{i_l} \circ \dots \circ f_{i_1}$ such that $G(a)>b$. Such a composition of length $l$ has a positive probability $p$ to be applied at every moment, including one of the moments when the image $F_n(x)$ enters $I$. The arguments below is a way of formalising the following idea. At each moment when $F_n(x)\in I$, the chance to apply $G$ is at least $p$, and if there is an infinity of such moments, there should be also an infinity of moments when $G$ is applied afterwards, bringing the image $F_{n+l}(x)=G(F_n(x))$ above~$b$.

To proceed formally, consider the conditional probabilities of the event $B= A_k$ 
with respect to the growing cylinders generated by the first $m$ applied maps $g_1,\dots,g_m$, 
\begin{equation}\label{prob-B}
\Pb(B \mid g_1,\dots,g_m).
\end{equation}
Due to a general statement from the measure theory, the \emph{conditional} probabilities of an event~$B$ w.r.t. a growing family of cylinders generating the $\sigma$-algebra converge to $0$ or to $1$ almost surely, and the probability of tending to $1$ equals to $\Pb(B)$. The convergence follows from the martingale convergence theorem (as such conditional probabilities form a martingale), and the values 0 or 1 follow from the fact that every event can be approximated by a cylindrical one up to an arbitrarily small measure. This statement is also an analogue of the statement that almost every point of a measurable set is its Lebesgue density point.

However, such conditional probability can never exceed $(1-p)$ (whatever the values of $k$, $m$ and $g_1,\dots,g_m$ are), as after the first time the iteration $F_n(x)$ visits $(m,+\infty)$ with $n\ge \max(m,k)$ the probability that the next $l$ applied maps correspond to the map $G$ is at least~$p$, and for every such $\om$ for the image $F_{n+l,\om}(x)$ we have 
\[
F_{n+l,\om}(x) = G(F_{n,\om}(x)) > G(a)>b,
\]
thus such $\om$ does not belong to~$A_k$.

Hence, the conditional probability~\eqref{prob-B} converges to zero almost surely, and hence $\Pb(B)=0$. This contradiction proves that the probability that the upper limit takes a value in any finite interval vanishes, and thus this limit is almost surely equal to $+\infty$ or $-\infty$. 

The second statement of the lemma is proved analogously.
\end{proof}

 In the statements of Theorems~\ref{th:class} and~\ref{th:stat} we assume the set of generating maps $\mathcal{F}$ to be finite. As we will see in Sec.~\ref{monster}, this finiteness assumption cannot be dropped completely; however, it can be weakened to the following one (it is easy to see that this is actually the assumption used in their proofs).
\begin{definition} \label{displace}
A random dynamical system, generated by a measure $\nu$ on $Homeo_+(\bbR)$, has \emph{compact displacement} property, if for any $x\in \bbR$ its image $\{f(x)| f\in \supp \mu\}$ is contained in some compact interval. 
\end{definition}
\begin{rem}
This property holds automatically if $\nu$ is supported on some compact in $Homeo_+(\bbR)$, where the space of homeomorphisms is equipped with the topology of uniform convergence on the compacts of both~$f$ and~$f^{-1}$. 
\end{rem}

The main means to study RDS we're going to use throughout the first half of the paper, is to look at the behaviour of the points. Therefore we introduce the following functions, which allow us to do it in simpler terms.

\begin{notation}
Let us define
$$
\Fi_+(x) := \Pb(\lim_{n\rightarrow  \infty} F_n(x) = +\infty),
$$
$$
\Fi_- (x):= \Pb(\lim_{n \rightarrow  \infty} F_n(x) = -\infty),
$$
$$
\Fi_0 (x):= 1 - \Fi_+(x) - \Fi_-(x).
$$
\end{notation}

The first and the second are the probabilities of the events 'the iterations of $x$ tends to $+\infty$', 'the iterations of $x$ tends to $-\infty$'. The third one is the probability that the images of $x$ do not tend neither to $+\infty$, nor $-\infty$, and due to the Lemma \ref{limits} this is the same as the probability of 
\[
\limsup_{n\to\infty}{F_n(x)} = +\infty, \quad \liminf_{n\to\infty}{F_n(x)} = -\infty
\]
(oscillation behaviour). For a finitely generated RDS, this is equivalent to 'there exist an interval that $F_n(x)$ visits inifinitely many times'. In the infinite case it is not true: in Section \ref{monster} we present a counter-example.

$\hat \Fi_+, \hat \Fi_-$ and  $\hat \Fi_0$ are defined in the same manner for $\hat \mu$.

Now we can reformulate Theorem \ref{th:class} in terms of $\Fi_{\pm, 0}$.

\begin{Th}\label{th:fi}
For a pair of forward and inverse RDS with shiftability one of the following is true (perhaps, after the change of coordinate $x \rightarrow -x$ and/or inchanging $\mu$ and $\hat \mu$):
\begin{enumerate}
\item $\Fi_+ \equiv 1$, $\hat \Fi_- \equiv 1$;
\item $\Fi_+ \equiv 1$, $\hat \Fi_0 \equiv 1$;
\item $\Fi_0 \equiv 1$, $\hat \Fi_0 \equiv 1$;
\item $\Fi_0 \equiv 0$, $\Fi_+$  and $\Fi_-$ are not constant, $\hat \Fi_0 \equiv 1$.
\end{enumerate}
\end{Th}

Finally, recall the definition of a stationary measure:
\begin{definition}
A measure $\nu$ on $\bbR$ is called \emph{stationary} for the RDS $\langle \mathcal{F}, \mu \rangle$ with finite $\mathcal{F}$ if
\begin{equation}\label{eq:stat}
\nu = \sum_{i=1}^k p_i (f_i)_* \nu
\end{equation}
where $f_*\nu$ is the push-forward of the measure $\nu$ by the map $f$ (that is, $(f_* \mu)(A)=\mu(f^{-1}(A)$ for all Borel sets $A$).
\end{definition}

This definition is naturally generalized for the random dynamics generated by some probability measure $\mu$ on $Homeo_+(\bbR)$: 
\begin{definition}
A measure $\nu$ is \emph{stationary} for the corresponding RDS, if 
$$
\nu = \int (f_* \nu)  d\mu(f),
$$
or, equivalently, if for any Borel set $A\subset \bbR$ one has
$$
\nu(A) = \int \nu(f^{-1}(A))  \, d\mu(f).
$$
\end{definition}

This definition is also equivalent to the invariance of the measure $\mu^\bbN \times \nu$ for the skew product over the one-sided Bernoulli shift, but we will not use this here.

\section{Properties of $\Fi_+$ and $\Fi_-$}\label{fis}

In this section we study properties of functions $\Fi_+$, $\Fi_-$ and $\Fi_0$ on their own, without any relation to inverse dynamics. The reasoning holds for both finite and infinite RDS with shiftability property.

First, note that $\Fi_+$ and $\Fi_-$ are monotonous. Indeed, if for some $\omega \in \Omega$ $F_n(x)$ goes to $+ \infty$, then (as all our homeomorphisms preserve orientation) for any $y>x$ and any $n \in \bbN$  its image $F_n(y) \ge F_n (x)$ and thus also tends to $+\infty$. So $\Fi_+$ is non-decreasing. Similarly, $\Fi_-$ is non-increasing.

Next proposition states that either every point tends to~$+\infty$ (or, similarly,~$-\infty$), or the probability to go there vanishes at~$-\infty$ (correspondingly, at~$+\infty$).

\begin{pro}\label{prop1}
If there exists $\varepsilon>0$ such that for all $x \in \bbR: \Fi_+(x) \geqslant \varepsilon$, then $\Fi_+(x)\equiv 1.$

Symmetrically, if for some $\varepsilon>0$ all  $x \in \bbR: \Fi_-(x) \geqslant \varepsilon$, then $\Fi_-(x)\equiv 1.$
\end{pro}

\begin{proof}
Consider the event $A\subset \Omega$, stating that the iterations starting from the initial point $x$ do not tend to~$+\infty$:
\[
A =\{\omega\in \Omega \mid F_{n,\om}(x)\not\to + \infty, \quad n\to \infty\}.
\]
Take the conditional probabilities of this event with respect to the growing cylinders $g_1,\dots,g_m$. On one hand, due to the Markovian property such conditional probability equals to the probability that the iterations of the image point $F_m(x)=g_m\circ \dots \circ g_1(x)$ do not tend to $+\infty$: 
\begin{equation}\label{not-plus}
\Pb(A\mid g_1,\dots,g_m) = 1-\Fi_+(F_m(x)) 
\end{equation}
On the other hand, in the same way as in the proof of Lemma~\ref{limits}, these probabilities converge to $0$ or to $1$ almost surely, and the probability of tending to $1$ equals to $\Pb(A)$.

Applying this, we see that the probability $\Pb(A\mid g_1,\dots,g_m) $
converges almost surely to $0$ or to $1$. However, it cannot converge to $1$, as the right hand side of~\eqref{not-plus} is at most $1-\eps$ due to the assumption. Hence (again, in the same way as in the proof of Lemma~\ref{limits}), the limit is almost surely equal to~$0$, and thus $\Pb(A)=0$ due to the martingale property.
\end{proof}

Until now, we haven't used shiftability in our reasoning. But the following statement shows that it is important: due to it, different points of $\bbR$ cannot show completely different behavioral patterns -- that is, if one can go to any of infinities, so do all of them.

\begin{Lem}\label{prop2}
If there exists $x \in \bbR$ such that $\Fi_+(x) > 0$, then for every $y \in \bbR$ $ \Fi_+(y)>0.$
Similarly, if there exists $x \in \bbR$ such that $\Fi_-(x) > 0$, then for every $y \in \bbR$ $ \Fi_-(y)>0.$
\end{Lem}
\begin{proof}
Fix $y$. Shiftability allows us to move $y$ farther to the right than $x$ with positive probability, say, $p$. If $y$ is already greater than $x$, we can skip this step and pose $p = 1$. But any point greater than $x$ goes to infinity with probability at least $\Fi_+(x)$, so 
$$
\Fi_+(y) \ge p\cdot \Fi_+(x)>0.
$$
\end{proof}

We then have the following 

\begin{pro}\label{lem1}
If there exists $x$ and $y$ such that $\Fi_+(x) > 0$ and $\Fi_-(y) > 0$, then for every $z \in \bbR$ $\Fi_+(z) + \Fi_-(z)=1$.
\end{pro}

\begin{proof}
Applying Lemma~\ref{prop2}, we see that in this case $\Fi_+(0), \Fi_-(0)>0$. Note now, that the function 
$\Fi_0(z)$ is thus bounded away from~1. Indeed, due to the monotonicity of $\Fi_{\pm}$ for $z\ge 0$ we have $\Fi_+(z)\ge \Fi_+(0)$, while for $z\le 0$ we have $\Fi_-(z)\ge \Fi_-(0)$, thus
\[
\forall z\in \bbR \quad \Fi_+(z)+\Fi_-(z) \ge \min (\Fi_+(0),\Fi_-(0))=:\eps>0,
\]
and hence
\begin{equation}\label{one-eps}
\forall z\in \bbR \quad \Fi_0 (z) =1-\Fi_-(z)-\Fi_+(z) \le 1-\eps.
\end{equation}

As in the proof of Proposition~\ref{prop1}, take any initial point $x\in\bbR$ and consider the conditional probabilities 
\[
\Pb(\lim_{n\to\infty} F_{n,\om}(x) \neq \pm\infty \mid g_1,\dots, g_m).
\]
On one hand, such a conditional probability is equal to~$\Fi_0(F_m(x))$ due to the Markovian property. On the other hand, it should (due to the same arguments) converge to $0$ or $1$, converging to $1$ with the probability $\Fi_0(x)$. However, due to uniform upper bound~\eqref{one-eps} it cannot converge to~$1$, hence $\Fi_0(x)=0$.

\end{proof}

Now we see, that we do not have much freedom with the behavior of the random iterations: at least one of the functions $\Fi_+$, $\Fi_-$ and $\Fi_0$ must vanish identically. The next proposition makes this observation even stronger:

\begin{pro}\label{lem2}
Either $\Fi_+(z) + \Fi_-(z) \equiv 1$, or $\Fi_+(z) + \Fi_-(z) \equiv 0$. Equivalently, either $\Fi_0 \equiv 0$ or $\Fi_0\equiv 1$.
\end{pro}

\begin{proof}
Assume that $\Fi_+>0$. As in the proof of Proposition~\ref{prop1}, take any initial point $x\in\bbR$ and consider the event $A=\{F_{n,\om}(x)\to+\infty\}$ and its conditional probabilities w.r.t. $g_1,\dots,g_m$. 

Again due to the same measure theory arguments the conditional probability
\begin{equation}\label{A:plus}
\Pb(A\mid g_1,\dots,g_m) = \Fi_+(F_m(x)) 
\end{equation}
converges as $m\to\infty$ almost surely to $0$ or to $1$, and tends to $1$ with the probability equal to $\Fi_+(x)$, hence to $0$ with the probability $1-\Fi_+(x)$.

Now, due to monotonicity of $\Fi_+$, if $\Fi_+(F_{m,\om}(x))\to 0$, then $F_{m,\om}(x)\to -\infty$. Hence, $\Fi_-(x)\ge 1-\Fi_+(x)$, and thus $\Fi_+ + \Fi_- \equiv 1$. The case $\Fi_->0$ is treated analogously, and $\Fi_+=\Fi_-\equiv 0$ implies $\Fi_0\equiv 1$.
\end{proof}

\section{Proof of the Theorem~\ref{th:fi}}\label{2dyn}

In the previous section we proved that either one of the functions $\Fi_+$, $\Fi_-$ and $\Fi_0$ is identically equal to $1$ (immediately forcing two others to vanish), or $\Fi_- = 1 - \Fi_+$ and both are monotonously approaching $0$ and $1$, though never reaching. This section is devoted to the duality arguments, relating possible behaviours for $\mu$ and $\hat{\mu}$. 

We start with the following proposition; it is quite natural to expect, if we think of the inverse dynamics as a dynamics with reverted time.

\begin{pro}
Suppose $\Fi_+ \equiv 1$. Then $\hat \Fi_+ \equiv 0$. Similarly, if $\Fi_- \equiv 1$, then $\hat \Fi_- \equiv 0$.
\end{pro}
\begin{proof}
Let us prove the first statement of the proposition.
Fix $x \in \bbR$. As $\Fi_+ (x) = 1$, 
\[
\forall y \in \bbR \quad \Pb(F_k(x)>y) \rightarrow 1, \quad \text{ as } k\rightarrow \infty. 
\]
Suppose there exists $y$ such that $\hat \Fi_+(y) = p$. Then 
\[\liminf_{n\rightarrow \infty} \Pb(\hat F_n(y) > z)\geqslant p. 
\]
Therefore there exists such $N \in \bbN$, that 
$$
\Pb(x>\hat F_n(y)) = \Pb(F_n(x)>y) > 1 - \frac{p}{2}
$$
and $\Pb(\hat F_n(y)>x) >  p/2$ simultaneously. This contradiction concludes the proof.
\end{proof}

So, if $\Fi_+ \equiv 1$ then either $\hat \Fi_- \equiv  1$ or $\hat \Fi_0 \equiv 1$. The first case is illustrated by asymmetrical random walk (sample functions $f_{1,2}(x) = x \pm 1$ with probabilities different from $1/2$). The second case is a little trickier, yet still realizable by our means (see Fig.\ref{f:maps}). Put

$$
f_1(x) = 
	\begin{cases}
		x + 1, & \text{if} ~~x<0;\\
		2x + 1, & \text{if} ~~x \geqslant 0;
	\end{cases}
	 ~~~p_1 = \frac{1}{2}.
$$
and
$$
f_2(x) = x - 1,~ p_2 = \frac{1}{2}.
$$

Now, note that the probability that the iterations $F_{n,\om}(x)$ starting with $x\ge 0$ tend to $+\infty$ is strictly positive. In particular, 
\[
\Fi_+|_{[0,+\infty)} \ge \Fi_+(0)>0.
\]
On the other hand, on $(-\infty,0]$ our RDS is just a standard ``$+1$/$-1$'' random walk, and hence the images of any point $x<0$ almost surely reach $[0,+\infty)$. Applying the Markov property, we get that $\Fi_+|_{(-\infty,0)}\ge \Fi_+(0)$, and hence the function $\Fi_+$ is bounded away from~$0$. By Proposition~\ref{prop1}, it implies $\Fi_+\equiv 1$.

On the other hand, the trajectories of the inverse RDS almost surely do not tend to infinity. Indeed, on the negative half-line we still have ``$+1$/$-1$'' random walk, while $+\infty$ (under a change of coordinates $z=\frac{1}{x}$) becomes a positive Lyapunov exponent point.

\begin{figure}[!h!]
\begin{center}
\includegraphics{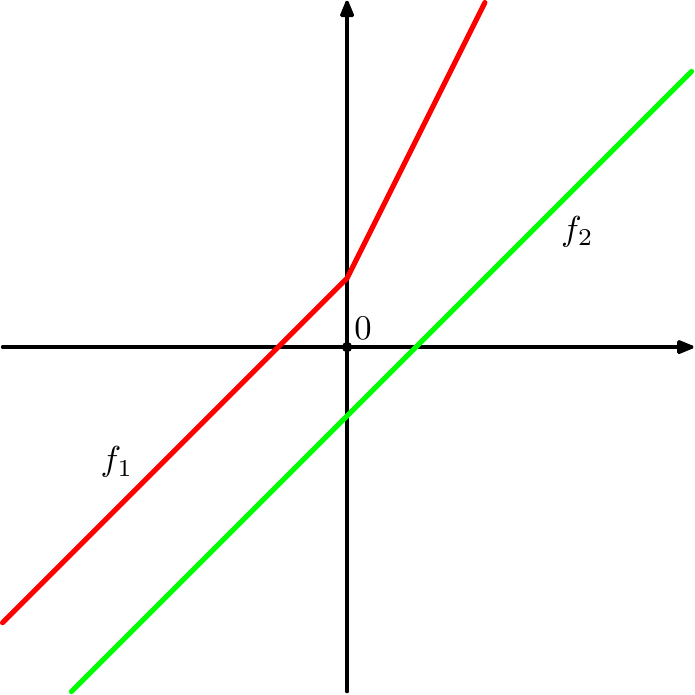}
\end{center}
\caption{Maps $f_1$ and $f_2$}\label{f:maps}
\end{figure}

Case with $\Fi_- \equiv 1$ becomes the one considered above under the change of coordinate $x \rightarrow -x$. Similarly, $\hat \Fi_{\pm} \equiv 1$ generate the same cases under the interchange of forward and inverse dynamics. All that rests are <<almost>> symmetrical cases: 

\begin{enumerate}
\item $\Fi_0 \equiv 1$, $\hat \Fi_0 \equiv 1$;
\item $\Fi_+$  and $\Fi_-$ are not constant, $\hat \Fi_0 \equiv 1$;
\item $\Fi_+$, $\Fi_-$, $\hat \Fi_+$ and $\hat \Fi_-$ are not constant. 
\end{enumerate}

Examples for the first two are quite simple to present: a classical random walk ($f_{1,2}(x) = x \pm 1$ with probabilities $1/2$) for the former, and the same random walk with additional function $f_3(x) = 2x$ with some positive probability for the latter. The third case, as it appears, never realizes.

In order to prove it, consider the following measure:
\begin{equation}
\nu[x,y] = \Fi_+(y) - \Fi_+(x).
\end{equation}

It is easy to check straightforwardly that
\[
\Fi_+(x) = \sum_{i = 1}^{k} \Fi_+(f_i(x)) \cdot p_i,
\]
(where $k$ can be infinite). Thus we conclude that $\nu$ is stationary for inverse dynamics. From the definition of $\Fi_+$ and our assumbtions we conclude that $\nu$ is stochastical and non-constant. 

Let us take an ergodic component $\tilde \nu$ of $\nu$; stochastical ergodicity theorem of Kakutani (\cite{KAK}, \cite[Theorem 3.1]{FUR}) then implies that for almost any starting point $t$ its random orbit is almost surely (asymptotically) distributed with accordance with $\tilde \nu$. In particular, it will visit arbitrarily many times a closed interval with any strictly positive measure. Therefore $\hat \Fi_0(t) = 1$, and then $\hat \Fi_0 \equiv 1$.

Thus we have proved the Theorem~\ref{th:fi}.

\section{Infinite monster} \label{monster}

One of the arguments in the finitely generated RDS case was that if a trajectory almost surely does not tend neither to $+\infty$, nor to $-\infty$, then it almost surely endlessly oscillates between the infinities, and thus visits a sufficiently large interval $J$ infinitely often. This section is devoted to construction of a ``monstrous'' example showing that this is no longer the case for infinitely generated systems. 

The idea is quite natural: if we want to make such a system whose orbits avoid any compact interval after some initial amount of time, we need the absolute value of $x$ to tend to $\infty$ and also allow sufficiently large ``jumps'', so the orbit could avoid getting ``caught'' in a finite interval. 
In order to do so, we consider strongly shifting maps:

\begin{equation}\label{maps}
f_{k} (x) = x + (-1)^k e^{e^k}, \quad  \quad k=1,2,\dots.
\end{equation}
with sufficiently slowly decreasing probabilities
\begin{equation}\label{probas}
p_{k} =  \frac{1}{k} - \frac{1}{k+1}=  \frac{1}{k(k+1)}, \quad k=1,2,\dots.
\end{equation}

\begin{Th}
The trajectories of the RDS, defined by~\eqref{maps} and~\eqref{probas}, almost surely visit any compact interval only finitely many times. The same holds if we replace the maps $f_{ k}$ by any maps
$\tilde f_{\pm k}$ such that the difference $\tilde f_{k}(x)-f_{k}(x)$ is bounded uniformly in~$k$ and~$x$.
\end{Th}

\begin{proof}
We will prove the conclusion of the theorem for the original maps $f_{\pm k}$; the reader will easily see that the same proof still works for the perturbed ones, too.

Let us call $k$ in $f_{k}$ its \emph{rank}; from~\eqref{maps}, we see that each of these maps is much more <<powerful>> than a composition of a lot of maps of lower rank. Let us now consider the sequence of highest applied ranks. Namely, let $k_n$ be the (random) sequence of ranks, and let $K_n$ denote the maximal rank appearing up to the $n$-th iteration:
\[
K_n:=\max_{n'\le n} k_{n'}.
\]
We then have the following lower estimate for the growth of these ranks: 
\begin{Lem}\label{l:K-n}
Almost surely for all $n$ sufficiently large one has $K_n>\sqrt{n}$.
\end{Lem}
\begin{proof}
The event $\{K_n<\sqrt{n}\}$ coincides with the event $\{k_1<\sqrt{n},\dots, k_n<\sqrt{n}\}$, and thus (due to the choice~\eqref{probas} of probabilities) has the probability 
\[
\Pb(K_n<\sqrt{n}) = (1-\frac{1}{[\sqrt{n}]})^n < e^{-\sqrt{n}}.
\]
The application of Borel--Cantelli Lemma thus concludes the proof.
\end{proof}

Now, once $K_n>\sqrt{n}$ and $n$ is sufficiently large, it is immediate from the definition~\eqref{maps} that the highest rank maps (and there is at least one of them) overpower at most $n-1$ lower ranking ones, shifting the initial point $x$ to the corresponding infinity.
\end{proof}

The example above is asymmetric. However, it can be modified to become symmetric. Namely, take the maps
\begin{equation}\label{maps-inv}
f_{\pm k} (x) = x \pm e^{e^k}, \quad  n\in \bbN.
\end{equation}
and associate them with the probabilities
\begin{equation}\label{probas-inv}
p_{\pm k} = \frac{1}{2} \cdot \left( \frac{1}{k} - \frac{1}{k+1} \right), \quad k=1,2,\dots.
\end{equation}

\begin{Th}\label{t:symm}
The trajectories of the RDS, defined by~\eqref{maps-inv} and~\eqref{probas-inv}, almost surely visit any compact interval only finitely many times. The same holds if we replace the maps $f_{\pm k}$ by any maps
$\tilde f_{\pm k}$ such that the difference $\tilde f_{\pm k}(x)-f_{\pm k}(x)$ is bounded uniformly in~$k$ and~$x$.
\end{Th}

\begin{proof}
Take $n_j$ to be the sequence of the moments when the new highest rank appears: $k_{n_j}>K_{n_j - 1}$. The key argument of the proof is the following: almost surely, aside of a finite number of initial steps, the next highest rank map appears \emph{before} the previous highest rank appears again (and thus the maps a chance to cancel each other). 

\begin{Lem}\label{l:j}
Almost surely, for all $j$ sufficiently large, the following two statements hold: 
\begin{itemize}
\item $\forall n=n_j +1,\dots n_{j+1} \quad k_n<k_{n_j}$.
\item $k_{n_j}>2^{j/3}$.
\end{itemize}
\end{Lem}

\begin{proof}
Fix some $j$, and let us consider the conditional distribution to given $n_j$ and to the ranks $k_1,\dots,k_{n_j}$. In the sequence $k_{n}, \, n>n_{j}$, let us consider the first $n'$ such that $k_{n'}\ge k_{n_j}$. Due to the choice of probabilities~\eqref{probas}, we have 
\begin{equation}\label{eq:cond-proba}
\Pb(k_{n'}=k_{n_j}) = \frac{1/(k_{n_j} (k_{n_j}+1))} {1/k_{n_j} } = \frac{1}{k_{n_j}+1}.
\end{equation}
At the same time, due to the same definition, we have 
\[
\Pb(k_{n'}\ge 2 k_{n_j}) = \frac{1}{2},
\]
what implies the second conclusion of the lemma due to the Law of Large Numbers. Now, we use this second conclusion to return to the first one, again applying Borel--Cantelli Lemma to the events ``$k_{n_j}>2^{j/3}$ and the first conclusion for $j$ does not hold''. Indeed, the series 
$\sum_j \frac{1}{2^{j/3}+1}$ converges, hence almost surely only a finite number of these events take place.
\end{proof}

Now, it is easy to see that Lemma~\ref{l:j} (together with the estimate of Lemma~\ref{l:K-n}) allows to conclude the proof of Theorem~\ref{t:symm}. Indeed, its first conclusion implies that for all sufficiently large $j$ the highest rank map is applied only once, and Lemma~\ref{l:K-n} implies that this rank is sufficiently high to overpower the composition of all the other maps.

\end{proof}

\section{Proof of Theorem~\ref{th:stat}} \label{measures}

Note first that the argument from~\cite[Theorem 5.1]{DKNP} allows to construct a (possibly, Radon) stationary measure for a recurrent dynamics even without the symmetry assumption. 
\begin{Lem}
A recurrent RDS on the $\bbR$ admits a Radon stationary measure.
\end{Lem}
\begin{proof}
Namely, let $J\subset \bbR$ be an interval such that for any initial point $x$ its random images $F_n(x)$ almost surely visit $J$ infinitely often. Take a compactly supported smooth function $\psi:\bbR\to [0,1]$, such that $\psi |_J\equiv 1$, and consider a random process of iterations that is stopped on each step at the point $x_n$ with the probability $\psi(x_n)$. 

Denote by $m_x$ the distribution of the stopping point for the process starting at the point~$x$; then $m_x$ depends continuously on~$x$. Thus it admits (via the usual Kryloff--Bogolyubov procedure) a stationary measure~$\nu_{\psi}$, that is by construction supported on~$\supp \psi$.

Finally, if $\supp \psi_1 \subset \{x\mid \psi_2(x)=1\}$, it is not difficult to check that multiplying the measure $\nu_{\psi_2}$ by $\psi_1(x)$ we obtain a measure $\psi_1(x)\nu_{\psi_2}$ that is a (non-probability) stationary measure for the $\psi_1$-process. Hence, taking a sequence of functions $\psi$ with larger and larger domains $\{\psi=1\}$, and normalizing the corresponding measures $\nu_{\psi}$ on $\nu_{\psi}(J)=1$, one gets a (non-probability) Radon stationary measure for the initial random walk.
\end{proof}

\begin{rem}
The constructed measure is not guaranteed to be fully supported or non-atomic. Actually, taking three maps 
\[
f_1(x)=x+1, \quad f_2(x)=x-1, \quad f_3(x)=x+\frac{1}{10} \sin 2\pi x
\]
with equal probabilities, one gets the dynamics for which Radon stationary measures will be supported on~$\bbZ$.
\end{rem}

The above argument allows to construct a stationary measure in the case~\ref{i:3} of Theorem~\ref{th:stat}. Now, to distinguish this case from the cases~\ref{i:2} and~\ref{i:4}, we will need the following two propositions. The first of them handles the case~\ref{i:4}:
\begin{pro}\label{p:finite}
Assume that the inverse dynamics of RDS is recurrent. Then there exists a finite stationary measure for the inverse dynamics if and only if for the forward dynamics both functions $\phi_+, \phi_-$ do not vanish (in other words, that all the points tend to each of $\pm\infty$ with positive probability).
\end{pro}
The second one handles the case~\ref{i:2}:
\begin{pro}\label{p:semi}
Assume that the inverse dynamics of RDS is recurrent. Then there exists a semi-infinite stationary measure for the  inverse dynamics $\hat \mu$, such that $\hat \mu([x,+\infty)<\infty$, if and only if for the forward dynamics the function $\phi_+\equiv 1$ (in other words, that all the points tend to~$+\infty$).
\end{pro}

\begin{proof}[Proof of Proposition~\ref{p:finite}]
The construction is much more straightforward. Namely, if the function $\hat \phi_+$ (and hence $\hat \phi_-$) is non-constant, then (as was done in Section~\ref{2dyn}) one can take 
$$
\hat \nu((-\infty,x])=\phi_+(x+0) \quad \forall x\in \bbR.
$$

In the other direction, assume that there exists a probability stationary measure $\hat \mu$ for the inverse dynamics. Then let us consider the function $\varphi(x):=\hat \mu((-\infty,x])$. The stationarity relation~\eqref{eq:stat} implies that 
$$
\varphi(x) = \hat \mu((-\infty,x]) =  \sum p_i (f_i^{-1})_* \varphi((-\infty, x]) =  \sum p_i \varphi(f_i(x)),
$$
hence the sequence $\varphi(F_n(x))$ forms a martingale. This martingale thus converges almost surely. Moreover, this martingale is bounded, hence the expectation of the limit is equal to its initial value. On the other hand, the only possible limit values are 0 and 1, as both upper and lower limit of the sequence of random iterations can be only $-\infty$ and $+\infty$ (see Lemma \ref{limits}). Hence, both probabilities of tending to $-\infty$ and to $+\infty$ are positive, and this concludes the proof. 
\end{proof}

\begin{rem} 
The arguments above actually show that for \emph{any} probability stationary measure $\mu$ for the inverse dynamics its partition function coincides with the probability $\phi_+$ that the point tends to $+\infty$ in the forward dynamics. The latter probability is well-defined, thus implying the uniqueness of the inverse stationary measure $\mu$. 

It is interesting to compare this argument to the one in the proof of \cite[Theorem 1]{CS}, as they are quite parallel. Indeed, the forward-dynamics probability that $F_n(a)>M$ (for a large fixed $M$) is the same as the probability of $F_n^{-1}(M)<a$; however, the authors of \cite{CS} use Birkhoff ergodic theorem instead of the martingale arguments to conclude.
\end{rem}

\begin{proof}[Proof of Proposition~\ref{p:semi}]
Assume first that $\phi_+\equiv 1$. Then, the random trajectory $F_n(x)$ of every initial point $x$ almost surely tends to $+\infty$, and thus the minimum $\min_n F_n(x)$ is almost surely finite.

Now, for every $y\in \bbR$ consider the probability 
\[
\psi_y(x):=\Pb(\exists n\ge 0 : \, F_n(x)<y) = \Pb(\min_{n\ge 0} F_n(x)<y).
\]
Note that for every $x>y$ it satisfies the full probability relation
\[
\psi_y(x)=\sum_i  p_i \psi_y (f_i(x)),
\]
while for $x<y$ it is identically equal to~$1$.

Consider now the measure $\hat \nu_y$, defined by 
\[
\hat \nu_y([x,+\infty)) = \psi_y(x).
\]
This measure satisfies the inverse dynamics stationarity relation on the subsets of~$(y,+\infty)$.

Now, normalize this measure so that the measure of $[0,+\infty)$ is equal to~$1$: take 
\[
\hat \mu_y := \frac{1}{\psi_y(0)} \hat \nu_y,
\]
and consider any weak accumulation point $\hat \mu$ of $\hat \mu_y$ as $y\to -\infty$. Any such limit point will be a stationary measure for the inverse dynamics, by construction finite on $[0,+\infty)$.

In the other direction, if there exists a semi-finite stationary measure $\hat \mu$ for the inverse dynamics, let us consider the function $\psi(x)=\hat\mu([x,+\infty))$. This function again leads to a positive martingale $\psi(F_n(x))$, that is now unbounded due to infiniteness of $\hat \mu$. 

However, a positive martingale still converges almost surely, and now the only its possible limit is~0 (as upper and lower limits of $F_n(x)$ can be only $+\infty$ or $-\infty$, and the function $\psi$ tends to infinity at $-\infty$). Hence, $\psi(F_n(x))$ converges to $0$ almost surely, and thus almost surely $F_n(x)\to+\infty$. Thus $\psi_+\equiv 1$.
\end{proof}


\begin{thebibliography}{0} 

\bibitem{AM} {\sc L. Alsed\`a, M. Misiurewicz}, Random interval homeomorphisms, \emph{Publ. Mat.} \textbf{58} (2014), 15--36.
 
 
\bibitem{BM} {\sc A. Bonifant and J. Milnor}, Schwarzian derivatives and cylinder maps, 
In: Holomorphic Dynamics and Renormalization, Fields Institute communications, v.~53, pp. 1--21. 
American Mathematical Soc., Providence, RI (2008).
 
\bibitem{BB}{\sc S.~Brofferio, D.~Buraczewski}, On unbounded invariant measures of stochastic dynamical systems, \emph{Ann. Probab.} Volume 43, Number 3 (2015), 1456-1492.

\bibitem{BBD}{\sc S.~Brofferio, D.~Buraczewski, E.~Damek}, On the invariant measure of the random difference equation $X_n = A_n X_{n-1} + B_n$ in the critical case, \emph{Ann. Inst. H. Poincaré Probab. Statist. 48 (2012)}, no. 2, 377--395. doi:10.1214/10-AIHP406.

\bibitem{Sar}{\sc S.~Brofferio, D.~Buraczewski, T.~Szarek}, On uniqueness of invariant measures for random walks on $\Homeo^+(\bbR$). arXiv:2008.01185v1

 
\bibitem{CzS} {\sc W. Czernous, T. Szarek} Generic invariant measures for iterated systems of interval homeomorphisms, \emph{Archiv der Mathematik} \textbf{114} (2020), pp.~445--455.

\bibitem{CS}{\sc K.~Czudek, T.~Szarek} Ergodicity and central limit theorem for random interval homeomorphisms, \emph{Isr. J. Math.} (2020). https://doi.org/10.1007/s11856-020-2046-4
 
\bibitem{DKNP} {\sc B.~Deroin, V.~Kleptsyn, A.~Navas, K.~Parwani}, Symmetric random walks on $Homeo^+(\bbR)$, \emph{Annals of Probability Vol. 41}, No. 3B (2013), 2066-2089
 
\bibitem{FUR}{\sc A.~Furman}, Random walks on groups and random transformations. \emph{Handbook of dynamical systems}, Vol. 1A, pp. 931--1014, North-Holland, Amsterdam, 2002.
 
 \bibitem{Homb} {\sc M.~Gharaei, A.~J.~Homburg}, Random interval diffeomorphisms, \emph{Discrete \& Continuous Dynamical Systems} - S, 2017, 10 (2) : 241-272. doi: 10.3934/dcdss.2017012
  
 \bibitem{GP}{\sc Y.~Guivarc'h, E.~Le Page}, Spectral gap properties for linear random walks and Pareto's asymptotics for affine stochastic recursions,  \emph{Ann. Inst. H. Poincaré Probab. Statist.
Volume 52}, Number 2 (2016), 503-574.

\bibitem{IKS} {\sc Yu. Ilyashenko, V. Kleptsyn, P. Saltykov},
Openness of the set of boundary preserving maps of an annulus with intermingled attracting basins,
\emph{Journal of Fixed Point Theory and Applications} \textbf{3} (2008), pp.~449--463

\bibitem{KAK}{\sc S.~Kakutani}, Random ergodic theorems and Markov processes with a stable distribution. \emph{Proceedings of the Second Berkeley Symposium on Mathematical Statistics and Probability}, 1950, University of California Press, Berkeley and Los Angeles (1951), pp. 247–261.

 \bibitem{Kan} {\sc I. Kan}, Open sets of diffeomorphisms having two attractors, each with an everywhere dense basin, \emph{Bull. Am. Math. Soc.}, \textbf{31} (1994), pp.68--74
 
 
 \bibitem{SZ} {\sc T.~Szarek, A.~Zdunik}, Attractors and invatiant measures for random interval homeomorphisms, unpublished manuscript.


\end{thebibliography}
\end{document}